\documentclass[graybox]{svmult}

\usepackage{type1cm}        
\usepackage{makeidx}         
\usepackage{graphicx}        
\usepackage{multicol}        
\usepackage[bottom]{footmisc}

\usepackage{newtxtext}       %
\usepackage{newtxmath}       

\makeindex             

\begin{document}

\title*{Generating from the Strauss Process using stitching}
\author{Mark Huber}
\institute{Mark Huber \at Claremont McKenna College, 850 Columbia AV, Claremont, CA  91711, \email{mhuber@cmc.edu}
}
%
%
\maketitle

\abstract{The Strauss process is a point process with unnormalized density with respect to a Poisson point process, where each pair of points within a specified distance $r$ of each other contributes a factor $\lambda \in (0, 1)$ to  the density.  Basic Acceptance Rejection works spectacularly poorly for this problem, which is why several other perfect simulation methods have been developed.  These methods, however, also work poorly for reasonably large values of $\lambda$.  \emph{Acceptance Rejection Stitching} is a new method that works much faster, allowing the simulation of point processes with values of $\lambda$ much larger than ever before.
}

\section{Introduction}
\label{SEC:intro}

The \emph{Strauss process} is a point process that has a penalized density with respect to an underlying Poisson point process.  Given a point space $S \subset \mathbb{R}^n$ of finite Lebesgue measure, say that the random variable $X \subset S$ is a \emph{point process} if it contains a finite number of points with probability 1.  Write $X = \{x_1, \ldots, x_n\}$ if it contains $n$ points.

A random point process $X$ is a \emph{Poisson point process of rate $\lambda$} if the number of points in $X$ has a Poisson distribution with mean equal to $\lambda$ times the Lebesgue measure of $S$, and given the number of points in $X$, each is uniformly distributed over $S$.  The method to be given works on Poisson point processes with more general rate functions, but for simplicity of presentation the rate will be assumed to be a constant $\lambda$ here.

For a point process $X \subset S$ and positive constant $r$, let $c_r(X)$ be the number of pairs of distinct points in $X$ that are at most distance $r$ apart.  For $\gamma \in [0, 1]$, let 
\begin{equation}
f_{\gamma, r}(x) = \gamma^{c_r(X)}.
\end{equation}

A point process with (unnormalized) density $f_{\gamma, r}$ with respect to the underlying measure that is a Poisson point process with rate $\lambda$ over $S$ is a Strauss process~\cite{strauss1975}.  Because $\gamma \leq 1$, this density penalizes point process that have many points within distance $r$ of each other.  This density can also be written as a product.
\begin{equation}
f(x) = \prod_{\{x_i, x_j\} \subseteq x} [\gamma \cdot \mathbb{I}(\text{dist}(x_i, x_j) \leq r) + 1 \cdot \mathbb{I}(\text{dist}(x_i, x_j) > r)].
\end{equation}
Here $\mathbb{I}$ is the usual indicator function that evaluates to 1 if the argument is true and is 0 otherwise.

Say that a density is a \emph{penalty factor density} if it consists of factors each of which is at most 1.  For such a density, the \emph{acceptance-rejection} (AR) method can be used to generate samples exactly from the target distribution.  Generate a sample from the original distribution.  Then, for each factor in the density, accept the result with probability equal to the factor.  If every factor is accepted, accept the overall sample as coming from the density.  Otherwise, start the process over.

\begin{programcode}{$\texttt{AR-Strauss}(S, \gamma, r, \lambda)$}
  
  1. Draw $X$ as a Poisson random variable with rate $\lambda$ over $S$.  Say $X$ has $n$ points.
  
  2. For each $1 \leq i < j \leq n$, generate $U_{i, j}$ uniform over $[0, 1]$.
  
  3. If for all $1 \leq i < j \leq n$,
\[
U_{i, j} \leq [\gamma \cdot \mathbb{I}(\text{dist}(x_i, x_j) \leq r) + 1 \cdot \mathbb{I}(\text{dist}(x_i, x_j) > r)],
\]
then return $X$.

  4. Else, let $Y$ be the output of a call to $\texttt{AR-Strauss}(S, \gamma, r, \lambda)$.  Return $Y$.
\end{programcode}

This method was the first \emph{perfect simulation} method for generating exactly from the Strauss process.  The general AR protocol goes back to~\cite{vonneumann1951}.  More recently, other perfect simulation methods for the Strauss process have been developed.  These include:
\begin{itemize}
  \item
    Dominated coupling from the past (DCFTP)~\cite{kendall1995, kendallt1999, kendallm2000}.
  \item Birth-death-swap with bounding chains (BDS)~\cite{huber2012a}.
  \item Partial rejection sampling (PRS)~\cite{jerrumg2019} (when $\gamma = 0$).
\end{itemize}
    
See~\cite{huber2011b} and~\cite{huber2015b} for more detail and the theory underlying these methods.  In particular, DCFTP, BDS, and PRS all rely on the process being \emph{locally stable}.  A density $f$ is locally stable if for any set of points $x$ and any point $a$ in $S$, there is a constant $K$ such that $f(x \cup \{a\}) \leq K f(x)$ (see~\cite{kendallm2000}.)  Approaches that require local stability will be referred to as \emph{local methods}.

The running time of AR tends to be exponential in $\lambda$ and the size of the point space $S$.  The running time of local methods tend to be polynomial in the size of $S$ when $\lambda$ lies below a certain threshold (the critical value) and then exponential above that threshold.  This makes generating from the Strauss process difficult for high values of $\lambda$.  

In this work, a new method for generating from the Strauss process is presented.  Like generic AR the new method has an exponential running time in $\lambda$, but the rate of exponential growth is much smaller in the size of the point space $S$.  Therefore, the rate of the exponential is much lower than both AR and local methods past their critical point.  

The result is an algorithm that allows generation of Strauss processes over $(\lambda, S)$ pairs that were computationally infeasible before.  For instance, Figure~\ref{FIG:strauss} illustrates such a process with $\lambda = 200$, $\gamma = 0$, and $r = 0.15$.

\begin{figure}[b]
  \sidecaption
  \includegraphics[scale=0.75]{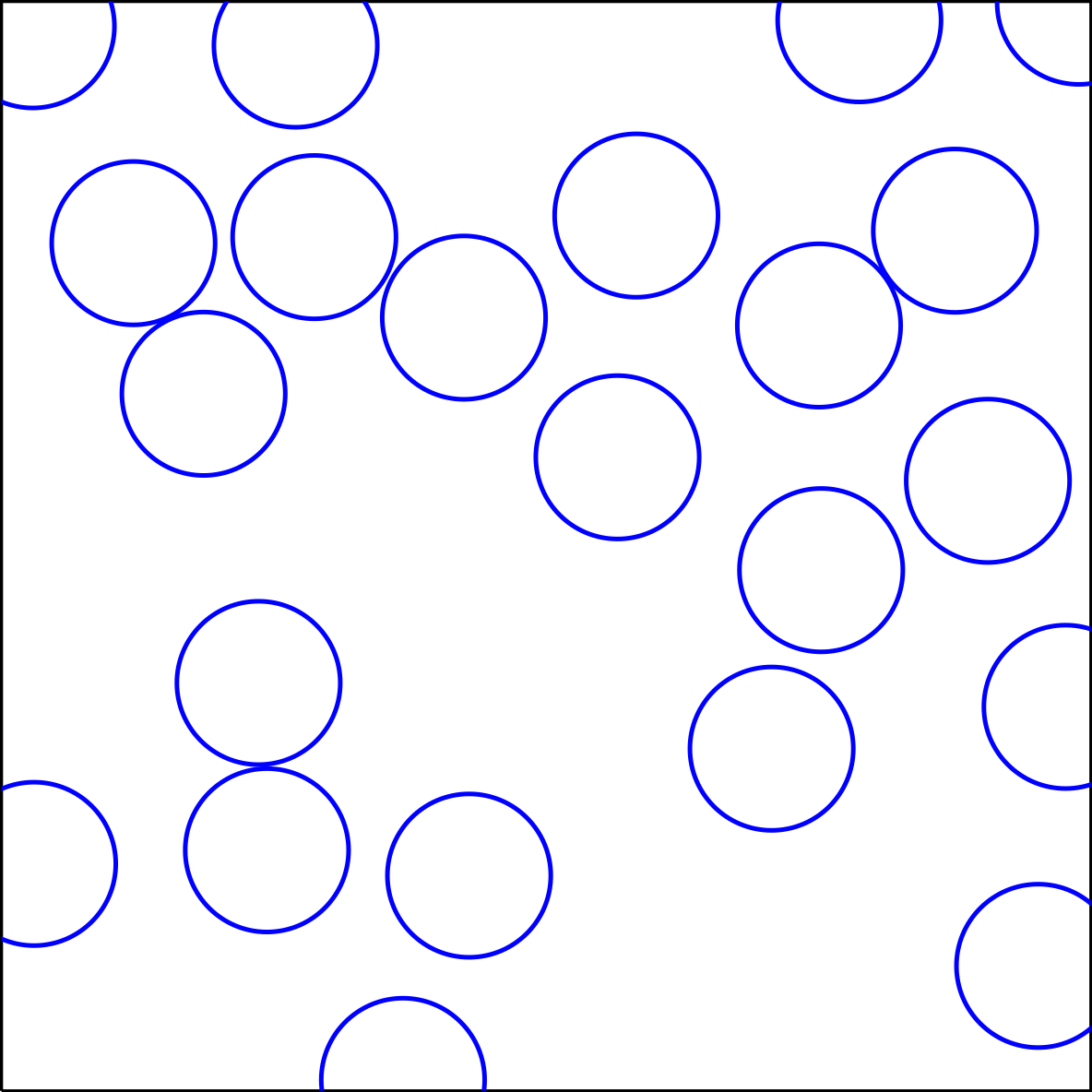}
  \caption{Strauss process over $S = [0, 1] \times [0, 1]$ with $\gamma = 0$, $\lambda = 200$, and $r = 0.15$.}
  \label{FIG:strauss}
\end{figure}

The rest of the paper is organized as follows.  The next section presents the new stitching algorithm, and presents results on correctness and running time.  Section 4 then gives numerical results on the running time.  Section 5 then concludes.

\section{Acceptance rejection and stitching}

Given an unnormalized penalty density $h_S \leq 1$ with underlying measure $\mu$, consider the general $\texttt{AR}$ algorithm for a point process where the points must lie in $S$.  This method begins by drawing a point process $X$ from the reference measure $\mu$.  Then with probability $h_S(X)$, $X$ is accepted as coming from density $h_S$ with respect to $\mu$.

Otherwise, the point process is rejected.  If rejection occurs, then recursion is used and the acceptance rejection algorithm calls itself to find the sample.  Of course, for practicality a while or repeat loop is used instead of recursive calls.  Here the algorithm is presented recursively, as this is a hallmark of perfect simultion algorithms, and gives a preview to later methods.

\begin{programcode}{$\texttt{AR}(h, \mu, S)$}

1. Draw $X$ from $\mu$ for point space $S$.

2. Draw $U$ uniformly from $[0, 1]$.

3. If $U \leq h(X)$, then return $X$.

4. Else, let $Y$ be the output of a recursive call to $\texttt{AR}(h, \mu, S)$.  Return $Y$.

\end{programcode}

For density $h$ and point space $S$, let $h_S(x) = h(x \cap S)$ be the density restricted to only consider points in $S$.

\begin{lemma}
Let $Z_h$ be the integral of $h_S$ with respect to $\mu$ with points in $S$.  If $Z_h$ is greater than zero, then the output of $\texttt{AR}(h, \mu, S)$ has density $h_S$ with respect to measure $\mu$ with points in $S$.  The number of times \texttt{AR} is called is geometrically distributed with mean $1 / Z_h$.
\end{lemma}

The proof uses \emph{The Fundamental Theorem of Perfect Simulation} (FTPS)~\cite{huber2015b} which gives two conditions under which the output of a probabilistic recursive algorithm $\mathcal{A}$ comes from a target distribution.  The first condition is that $\mathcal{A}$ must terminate with probability 1.

Now consider an algorithm $\mathcal{A'}$ where the recursive calls is $\mathcal{A}$ are replaced with oracles that generate from the correct distribution.  If $\mathcal{A'}$ has output that provably comes from the correct distribution, say that $\mathcal{A'}$ is \emph{locally correct}.  The second condition in the FTPS is that $\mathcal{A'}$ is locally correct.

\begin{proof}
First consider the probability that the algorithm accepts.  
\begin{align*}
\mathbb{P}(U \leq h_S(X)) &= \mathbb{E}[\mathbb{I}(U \leq h_S(X))] \\
&= \mathbb{E}[\mathbb{E}[\mathbb{I}(U \leq h_S(X)) | X]] \\
&= \mathbb{E}[h_S(X)] \\
&= \int  h_S(x) \ d\mu(x) = Z_h.
\end{align*}
By assumption this integral value $Z_h$ is greater than zero.  Hence the number of times the algorithm generates $X$ is a geometric random variable with a positive parameter, and so is finite with probability 1.

Now consider algorithm $\mathcal{A'}$, where in the last line the recursive call for $Y$ is replaced by an oracle.  Let $W$ be the output of $\mathcal{A'}$.  Then for any measurable set $A$, 
\begin{align*}
\mathbb{P}(W \in A) &= \mathbb{P}(X \in A, U \leq h_S(X)) + \mathbb{P}(U > h_S(X), Y \in A) \\
&= \frac{Z_h}{Z_h} \int_{x \in A} h_S(x) \ d\mu(x) + \mathbb{P}(U > h_S(X)) \mathbb{P}(Y \in A) \\
&= Z_h \mathbb{P}(Y \in A) + (1 - Z_h) \mathbb(Y \in A) = \mathbb{P}(Y \in A). 
\end{align*}
Therefore $\mathcal{A'}$ has the correct output distribution, making it locally correct.  Since the algorithm also terminates in finite time with probability 1, by the FTPS the original algorithm is also correct.
\end{proof}

Now suppose for $(S_1, S_2)$ a partition of $S$ that our target density can be factored into three parts.  The first part only depends on points in $S_1$, the second part only depends on points in $S_2$, and the third part only depends on interactions between a point in $S_1$ and a point in $S_2$.  That is, 
\[
h_S(x) = h_{S_1}(x \cap S_1) h_{S_2}(x \cap S_2) h_{S_1, S_2}(x),
\]
where $h_{S_1, S_2}(x)$ is also a penalty density.  Then it is possible to use the partition to create a faster algorithm.  Use \texttt{AR} to find samples from each half of the partition.  Then accept the combined result as a draw from the target distribution.

\begin{programcode}{$\texttt{AR-split-once}(h, \mu, S)$}

1. Partition $S$ into $(S_1, S_2)$.

2. Draw $X_1$ using $\texttt{AR}(h, \mu, S_1)$, draw $X_2$ using $\texttt{AR}(h, \mu, S_2)$.

3. Draw $U$ uniformly from $[0, 1]$.

4. If $U \leq h_{S_1, S_2}(X_1 \cup X_2)$ then return $X_1 \cup X_2$.

5. Else draw $Y$ from $\texttt{AR-split-once}(h, \mu, S)$.

\end{programcode}

\begin{lemma}
Let $Z_h$ be the integral of $h_S$ with respect to $\mu$ with points in $S$.  If $Z_h > 0$, then the output of $\texttt{AR-split}(h, \mu, S)$ has density $h_S$ with respect to measure $\mu$ with points in $S$.
\end{lemma}

\begin{proof}
Let $X_1$ be a draw from $\mu$ over $S_1$ and $X_2$ a draw from $\mu$ over $S_2$.  For $U_1, U_2, U$ independent uniforms over $[0, 1]$, let
\begin{align*}
p_1 &= \mathbb{P}(U_1 \leq h_{S_1}(X_1)) \\
p_2 &= \mathbb{P}(U_2 \leq h_{S_2}(X_2)) \\
p_3 &= \mathbb{P}(U \leq h_{S_1, S_2}(X_1, X_2) \mid U_1 \leq h_{S_1}(X_1), U_2 \leq h_{S_2}(X_2))).
\end{align*}

Then the chance of accepting $X$ as a draw from $h$ in line 4 is $p_1 p_2 p_3 = Z_h > 0$.  Hence $p_1$, $p_2$, and $p_3$ are positive, which means the calls to \texttt{AR} take on average a finite number of steps.  The number of calls to \texttt{AR-split-once} will on average be $1 / p_3$.  Taken together, this means that the algorithm terminates with probability 1 in finite time.

Now consider the output of $\mathcal{A'}$, where recursive calls are replaced with oracles.  Then $(X_1, X_2)$ has density $f_{(X_1, X_2)}(x_1, x_2) = h_{S_1}(x_1) h_{S_2}(x_2)$, and $Y$ from line 5 will have density $h(x)$.  Let $s$ be the event that acceptance occurs in line 5, that is
\begin{equation}
s = \left\{ U \leq h_{S_1, S_2}(X_1, X_2) \right\}.
\end{equation}

Let $p$ be the probability of $s$, then 
\begin{align*}
  p &= \mathbb{P}(U \leq h_{S_1,S_2}(X_1, X_2)) \\
  &= \int_{(x_1, x_2)} h_{S_1}(x_1)h_{S_2}(x_2)h_{S_1, S_2}(x_1, x_2) d[\mu(x_1) \times \mu(x_2)] \\
  &= Z_h.
\end{align*}
Let $W$ be the output of the algorithm.  Then for any measurable set $A$,
\begin{equation}
  \mathbb{P}(W \in A) = \mathbb{P}(X_1 \cup X_2 \in A, s) + \mathbb{P}(\neg s, Y \in A).
\end{equation}

The first term is the probability of accepting a draw that happens to fall in $A$, and the second term is the probability of not accepting and the recursive call generating output that lies in the target set $A$.  Since $s$ and $Y \in A$ are independent events and $\mathbb{P}(\neg s) = 1 - Z_h$,
\begin{equation}
  \mathbb{P}(W \in A) = \mathbb{P}(X_1 \cup X_2 \in A, s) + (1 - Z_h) \mathbb{P}(Y \in A).
\end{equation}

Further,
\begin{align*}
  \mathbb{P}(X_1 \cup X_2 \in A, s) &= \int_{(x_1, x_2) \in A} h_{S_1}(x_1)h_{S_2}(x_2) h_{S_1, S_2}(x_1, x_2) \ d[\mu(x_1) \times \mu(x_2)] \\
  &= Z_h \int_{(x_1, x_2) \in A} h_{S_1}(x_1)h_{S_2}(x_2) h_{S_1, S_2}(x_1, x_2) / Z_h \ d[\mu(x_1) \times \mu(x_2)] \\
  &= Z_h \mathbb{P}(Y \in A).
\end{align*}

Therefore $\mathbb{P}(W \in A) = Z_h \mathbb{P}(Y \in A) + (1 - Z_h) \mathbb{P}(Y \in A) = \mathbb{P}(Y \in A)$, and $\mathcal{A'}$ has the correct output distribution.  By the FTPS so does $\texttt{AR-split-once}.$
\end{proof}

\subsection{When to split more than once}

Using \texttt{AR}, the probability $p$ of acceptance is $p = p_1 p_2 p_3$.  So the expected number of times the random variable $X$ is sampled in $\texttt{AR}$ is 
\[
\frac{1}{p_1} \cdot \frac{1}{p_2} \cdot \frac{1}{p_3}.
\]

What is the running time of \texttt{AR-split-once}?  The call to $\texttt{AR-Split}(h, \mu, S_1)$ uses $1 / p_1$ draws on average from $X$, and the call to $\texttt{AR-Split}(h, \mu, S_2)$ uses $1 / p_2$.  These calls are repeated an average of $1 / p_3$ times.  Hence the expected number of times $X$ is sampled in $\texttt{AR-split}$ is 
\[
\left[\frac{1}{p_1} + \frac{1}{p_2} \right] \frac{1}{p_3}.
\]

Adding rather than multiplying $1 / p_1$ and $1 / p_2$ gives the speedup.  Also, recursion instead of AR should be used whenever $p_1 + p_2 > 1$.  This gives rise to the \emph{stitching} algorithm, which uses recursion as much as possible, in an adapted manner.

\begin{programcode}{$\texttt{AR-stitch}(h, \mu, S)$}

1. Draw $Z$ from $\mu$ with point space $S$, and independently draw $U_1$ uniform over $[0,1]$.  If $U_1 \leq h_S(Z)$, then return $Z$ and quit.

2. Partition $S$ into $(S_1, S_2)$.

3. Draw $X_1$ using $\texttt{AR-stitch}(h, \mu, S_1)$, draw $X_2$ using $\texttt{AR-stitch}(h, \mu, S_2)$.

4. Draw $U_2$ uniformly from $[0, 1]$.

5. If $U_2 \leq h_{S_1, S_2}(X_1 \cup X_2)$ then return $X_1 \cup X_2$.

6. Else draw $Y$ from $\texttt{AR-stitch}(h, \mu, S)$.

\end{programcode}

The first result is that this procedure terminates in finite time with probability 1 if $Z_h > 0$, no matter how the partitioning is done.

\begin{lemma}
Let $Z_h$ be the integral of $h_S$ with respect to $\mu$ with points in $S$.  If $Z_h > 0$, then $\texttt{AR-stitch}(h, \mu, S)$ terminates in finite time with probability 1 regardless of the choice of partition at line 2.
\end{lemma}

\begin{proof}
Let $r(p)$ be the supremum over the expected number of times $X$ is generated over all choices of $S_1, S_2$, and $h$ when the probability $X$ is accepted in line 1 is $p$.  Our goal will be to bound $r(p)$ in terms of $p > 0$.

Let $p_1$ be the probability that $X$ is accepted in the recursive call over $S_1$, $p_2$ the acceptance probability over $S_2$, and $p_3$ the probability that $(X_1, X_2)$ is accepted in line 4.  

As seen earlier, $p = p_1 p_2 p_3$.  There is always at least one draw of $X$ in any call, followed by a $1 - p$ chance of two recursive calls, followed by a $1 - p_3$ chance of a third recursive call.  Hence
\begin{equation}
r(p) = 1 + (1 - p)[r(p_1) + r(p_2) + (1 - p_3)r(p)].
\end{equation}

This holds for all $p' \geq p$, so letting $w = \sup_{p' \in [p, 1]} r(p)$ gives
\begin{equation}
w \leq 1 + (1 - p)[3w].
\end{equation}
An easy calculation then gives for $p \geq 3 / 4$, $r(p) \leq w \leq 4$.

This forms the base case for an induction proof of the following fact:  For all $i \in \{0, 1, 2, \ldots\}$, if  $p \geq (3 / 4)(1 - p)^i$, then $r(p)$ is finite.

Consider the induction step:  suppose for all $p \geq (3 / 4)(1 - p)^i$, there is finite $M$ such that $r(p) \leq M$.  Consider $i + 1$, and assume $p \geq (3 / 4)(1 - p)^{i + 1}$.

If $p_1 \geq (3 / 4)(1 - p)^i$ and $p_2 \geq (3 / 4)(1 - p)^i$, then
\begin{equation}
  r(p) \leq 1 + (1 - p)[M + M + (1 - p_3)r(p)],
\end{equation}
and $r(p) \leq (1 + (1 - p)2M) / (p + p_3 - p p_3)$, completing the induction in this case.

Note that if $p_3 < 1 - p$, then $p_1 > p / p_3 = (3 / 4)(1 - p)^i$ and $p_2 > p / p_3 = (3 / 4)(1 - p)^i$ and so the induction step also holds in this case.

It cannot hold that both $p_1$ and $p_2$ are less than $(3 / 4)(1 - p)^i$, as that would make $p < (9 / 16)(1 - p)^{2i} < (3 / 4)(1 - p)^{i + 1}$.  Hence the only case that remains to consider is if $p_3 > 1 - p$ and exactly one of $p_1$ or $p_2$ (say $p_1$ without loss of generality) is less than $(3 / 4)(1 - p)^i$.

If $p_3 > 1 - p$, then $(1 - p_3) < p$, and by the induction hypothesis
\begin{equation}
  r(p) \leq 1 + (1 - p)[r(p) + M + p r(p)],
\end{equation}
which gives
$r(p) \leq (1 + (1 - p)M) / p^2$, completing the induction.

Since $p > 0$, there is some $i$ such that $p \geq (3 / 4)(1 - p)^i$, and so $r(p)$ is finite for all $p > 0$.
\end{proof}

\begin{lemma}
Let $Z_h$ be the integral of $h_S$ with respect to  $\mu$ for points in $S$.  If $Z_h > 0$, then $\texttt{AR-stitch}(h, \mu, S)$ terminates in finite time with probability 1 with output distributed as unnormalized density $h_S$ with respect to $\mu$ over $S$.
\end{lemma}

\begin{proof}
The algorithm terminates with probability 1 by the previous lemma.  Hence by the FTPS, it is only necessary to show that the algorithm is locally correct.  

Let $\mathcal{A'}$ be the algorithm where lines 3 and 6 are replaced with oracles drawing from the correct distributions.  In particular, $Y$ is a draw from $\mu$ restricted to point space $S$.  For any measurable $B$, note 
\[
\int_{B} h_S(x) \ d\mu(x) = \frac{Z_h}{Z_h} \int_{B} h_S(x) \ d\mu(x) = Z_h \mathbb{P}(Y \in B).
\]

Fix a measurable set $A$, and let $W$ be the output of $\mathcal{A'}$.  Then the chance the output is in $A$ can be broken down into the probability of three events $e_1$, $e_2$, and $e_3$, representing acceptance at line 1, acceptance at line 5, or rejection and $Y \in A$ respectively.  That is, $\mathbb{P}(W \in A) = \mathbb{P}(e_1) + \mathbb{P}(e_2) + \mathbb{P}(e_3)$ where
\begin{align*}
e_1 &= \left(Z \in A, U_1 \leq h_S(Z)\right) \\
e_2 &= \left(U_1 > h_S(Z), X_1 \cup X_2 \in A, U_2 \leq h_S(X_1 \cup X_2)\right) \\
e_3 &= \left(U_1 > h_S(Z), U_2 > h_S(X_1 \cup X_2), Y \in A\right).
\end{align*}

As in the earlier proof of the correctness of acceptance rejection,
\begin{equation}
\mathbb{P}(e_1) = \mathbb{P}(Y \in A) Z_h.
\end{equation}

The chance that $Z$ is not accepted is 
\begin{equation}
\mathbb{P}(U > h_S(Z)) = 1 - \mathbb{P}(U \leq h_S(Z)) 
  = 1 - \int h_S(x) \ d\mu 
  = 1 - Z_h.
\end{equation}

Since $(X, U_1)$ and $(X_1, X_2, U_2)$ are independent:
\begin{align*}
\mathbb{P}(e_2) &= \mathbb{P}(U_1 > h_S(Z)) \mathbb{P}(X_1 \cup X_2 \in A, U_2 > h_S(X_1 \cup X_2)) \\
&= (1 - Z_h)\int_{x_1 \cup x_2 \in A} h_{S_1}(x_1) h_{S_2}(x_2) h_{S_1, S_2}(x_1, x_2) \ d\mu \\
&= (1 - Z_h)\int_{x_1 \cup x_2 \in A} h_{S}(x_1 \cup x_2) \ d\mu \\
&= (1 - Z_h)Z_h \mathbb{P}(Y \in A).
\end{align*}

Similarly, using independence of the pieces of the last term, the chance that we reject twice and then the recursive call lands in $A$ is 
\begin{equation}
\mathbb{P}(e_3) = (1 - Z_h)(1 - Z_h) \mathbb{P}(Y \in A).
\end{equation}

Putting these terms together gives
\begin{align*}
  \mathbb{P}(W \in A) &= \mathbb{P}(Y \in A)[Z_h + (1 - Z_h)Z_h + (1 - Z_h)^2]\\
&= \mathbb{P}(Y \in A)
\end{align*}
which completes the proof of correctness.
\end{proof}

In some cases, it is possible to know when $h$ is easy to sample from using AR, at which point, one can substitute basic AR in for line 1.  For instance, in \texttt{AR} for the Strauss process, acceptance occurs with probability 1 when there are no points in the draw.  Hence for $S$ small enough that $\lambda S < 1$, there is at least an $\exp(-1)$ chance of accepting.  The criterion for what is easy will vary from problem to problem.

\begin{programcode}{$\texttt{AR-stitch-base}(h, \mu, S)$}

1. For $(h, S)$ easy, draw $Z$ using $\texttt{AR}(h, \mu, S)$.  Return $Z$.

2. Partition $S$ into $S_1$ and $S_2$.

3a. Draw $X_1$ using $\texttt{AR-stitch-base}(h, \mu, S_1)$.

3b. Draw $X_2$ using $\texttt{AR-stitch-base}(h, \mu, S_2)$.

4. Draw $U_2$ uniformly from $[0, 1]$.

5. If $U_2 \leq h_{S_1, S_2}(X_1 \cup X_2)$ then return $X_1 \cup X_2$.

6. Else draw $Y$ from $\texttt{AR-stitch-base}(h, \mu, S)$.

\end{programcode}

\begin{lemma}
Let $Z_h$ be the integral of $h_S$ over $\mu$ on point space $S$.  If $Z_h > 0$, then $\texttt{AR-split-base}(h, \mu, S)$ terminates in finite time with probability 1 with output distributed as unnormalized density $h$ with respect to $\mu$ over $S$.
\end{lemma}

\begin{proof}
The proof follows the same outline as for $\texttt{AR-stitch}(h, \mu, S)$.
\end{proof}

\section{Stitching in practice}

To illustrate stitching in practice, consider the Strauss process and the Ising model.

\subsection{The Strauss process}

The Strauss density is determined by the parameters $S$, $\lambda$, $r$, and $\gamma$.

Given a process $X_1$ over $S_1$ and $X_2$ over $S_2$, $h_{S_1, S_2}(X_1, X_2)$ is $\gamma$ raised to the number of pairs of points $(x_1, x_2) \in X_1 \times X_2$ that are within distance $r$ of each other.  This gives the following algorithm.

\begin{programcode}{$\texttt{Strauss-AR-stitch-base}(\lambda, r, \gamma, S)$}

1. If $\lambda$ times the Lebesgue measure of $S$ is at most 5, draw $Z$ using AR, and return $Z$.

2. Partition $S$ into $S_1$ and $S_2$.

3a. Draw $X_1$ using $\texttt{Strauss-AR-stitch-base}(\lambda, r, \gamma, S_1)$.

3b. Draw $X_2$ using $\texttt{Strauss-AR-stitch-base}(\lambda, r, \gamma, S_2)$.

4. Draw $U_2$ uniformly from $[0, 1]$.  Let $c$ be the number of $a_i \in X_1$ and $b_j \in X_2$ such that $\text{dist}(a_i, b_j) \leq r$.

5. If $U_2 \leq \gamma^c$ then return $X_1 \cup X_2$.

6. Else return a draw from $\texttt{Strauss-AR-stitch-base}(\lambda, r, \gamma, S)$.

\end{programcode}

\subsection{The Ising model}

In the Ising model (and its extension, the Potts model), each vertex of a graph is given a label from a color set.  In the ferromagnetic model, edges of the graph are penalized by $\exp(-2\beta)$ (for $\beta > 0$ a constant) when the two edges of the set are colored differently.  The reference measure is uniform over all colorings of the vertices.

In other words, the density for a graph with edge set $E$ is
\begin{equation}
f(x) = \prod_{\{i, j\} \in E} [\exp(-2\beta)\mathbb{I}(x(i) \neq x(j)) + \mathbb{I}(x(i) = x(j))]
\end{equation}
with respect to counting measure over all colorings of the vertices of the graph.

A partition of a vertex set of a graph is called a \emph{cut}.  The stitching needs only check edges which \emph{cross the cut}, meaning that the endpoints of the edges lie in different halves of the cut.  The density (and reference measure) are determined by $\beta$, the edge set $E$, and the vertex set $V$.

\begin{programcode}{$\texttt{Ising-AR-stitch-base}(\beta, E, V)$}

1. If $V = \{v\}$, then choose $X(v)$ uniformly from the set of colors, return $X$.

2. Partition $V$ into $V_1$ and $V_2$.

3a. Draw $X_1$ using $\texttt{Ising-AR-stitch-base}(\beta, E, V_1)$.

3b. draw $X_2$ using $\texttt{Ising-AR-stitch-base}(\beta, E, V_2)$.

4. Draw $U_2$ uniformly from $[0, 1]$.  Let $c$ be the number of $i \in V_1$ and $j \in V_2$ such that $x(i) \neq x(j)$.

5. If $U_2 \leq \exp(-2\beta c)$ then return $(X_1, X_2)$.

6. Else return a draw from $\texttt{Ising-AR-stitch-base}(\beta, E, V)$.

\end{programcode}

\section{Numerical results}

For the Ising model, there are effective methods to perfectly sample for $\beta$ above and below the critical temperature~\cite{proppw1996}, so only the Strauss process is considered here.

In order to evaluate the running time behavior of various algorithms for generating from the Strauss process, timings were run on $S = [0, 1] \times [0, 1]$ for basic AR, the PRS method of~\cite{jerrumg2019}, and stitching.  AR is always exponential in $\lambda$, while PRS stays polynomial in $\lambda$ before moving to exponential in $\lambda$ past a certain threshold.  Acceptance Rejection Stitching (represented as ARS in the figure) is also exponential in $\lambda$, but at a much smaller rate.

\begin{figure}[b]
  \sidecaption
  \includegraphics[scale=0.35]{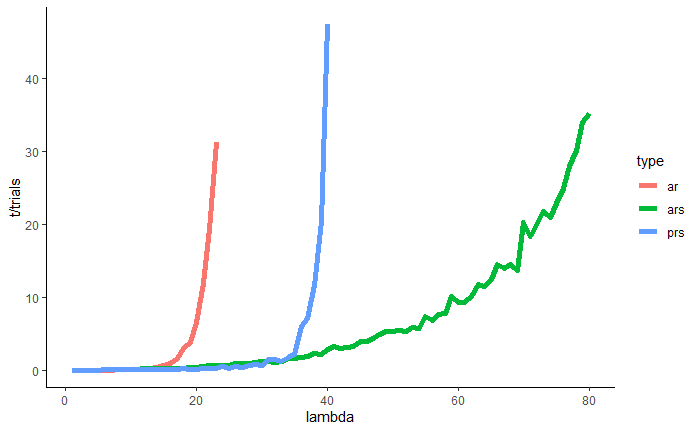}
  \caption{Timings of Acceptance Rejection, Partial Rejection Sampling, and Acceptance Rejection with Stitching for varying $\lambda$ over $S = [0, 1] \times [0, 1]$.}
\end{figure}

A plot of the log of the timings shows the exponential nature of the growth.  The original \texttt{AR} aalgorithm quickly becomes exponential in $\lambda$, while PRS stays polynomial until the critical point where it switches over to exponential behavior.  ARS also appears to be polynomial before turning exponential, but the slope of the log line is much lower than that of AR and PRS, allowing for sampling from much higher values of $\lambda$.

\begin{figure}[t]
  \sidecaption[t]
  \includegraphics[scale=0.35]{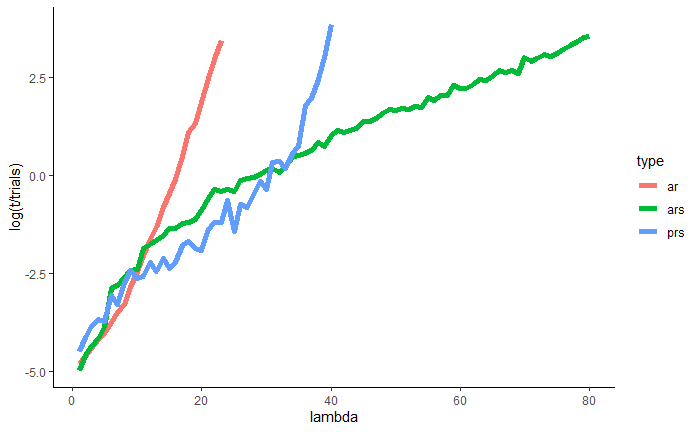}
  \caption{Log timings of Acceptance Rejection, Partial Rejection Sampling, and Acceptance Rejection with Stitching for varying $\lambda$ over $S = [0, 1] \times [0, 1]$.}
\end{figure}

\subsection{Code}

The code was written in R.  First the \texttt{tidyverse} library is needed.
\begin{verbatim}
library(tidyverse)
\end{verbatim}
Next, count the number of pairs of points within distance $r$ of each other.
\begin{programcode}{count\_r}
\begin{verbatim}
count_r <- function(points, r) {
  if (nrow(points) == 0) return(0)
  p <- points %>% mutate(k = 1)
  close <- 
    p %>%
    full_join(p, by = "k") %>%
    mutate(dist = sqrt((x.x - x.y)^2 + (y.x - y.y)^2)) %>%
    filter(dist > 0) %>%
    select(-k) %>%
    filter(dist < r)
  return(nrow(close) / 2)
}
\end{verbatim}
\end{programcode}

This allows implementation of basic AR for the Strauss process.
\begin{programcode}{strauss\_ar}
\begin{verbatim}
# Generate draws from the Strauss process
strauss_ar <- function(s_1, s_2, lambda, r, gamma) {
  # As noted earlier, use a repeat loop rather than 
  # recursion in practice
  repeat {
    n <- rpois(1, s_1 * s_2 * lambda)
    ppp <- tibble(
      x = runif(n) * s_1,
      y = runif(n) * s_2
    )
    if (runif(1) < gamma^counts_r(ppp, r))
      return(ppp)
  }
}
\end{verbatim}
\end{programcode}

For stitching, given two sets of points \texttt{p1} and \texttt{p2}, we need to count the number of pairs of points, one from \texttt{p1} and one from \texttt{p2} that lie within $r$ of each other.
\begin{programcode}{count\_r\_two}
\begin{verbatim}
count_r_two <- function(p1, p2, r) {
  if (nrow(p1) * nrow(p2) == 0) return(0)
  r2 <- p2 %>% mutate(k = 1)
  close <- 
    p1 %>%
    mutate(k = 1) %>%
    full_join(r2, by = "k") %>%
    mutate(dist = sqrt((x.x - x.y)^2 + (y.x - y.y)^2)) %>%
    filter(dist > 0) %>%
    select(-k) %>%
    filter(dist < r)
  return(nrow(close))
}
\end{verbatim}
\end{programcode}

Using \texttt{strauss\_ar} to draw samples when $\lambda$ times the size of $S$ is at most 5, we can draw from the Strauss process using \texttt{Strauss-AR-stitch-base}.
\begin{programcode}{strauss\_ars}
\begin{verbatim}
strauss_ars <- function(s_1, s_2, lambda, r, gamma) {
  if (s_1 * s_2 * lambda <= 5) {
    return(strauss_ar(s_1, s_2, lambda, r, gamma)) 
  }
  repeat {
    flip <- (s_1 < s_2)
    a_1 <- s_1 * (1 - flip) + s_2 * flip
    a_2 <- s_1 * flip + s_2 * (1 - flip)
    p1 <- strauss_ars(a_1 / 2, a_2, lambda, r, gamma) 
    temp <- strauss_ars(a_1 / 2, a_2, lambda, r, gamma) 
    ifelse (nrow(temp) == 0, 
            p2 <- temp, 
            p2 <- temp %>% mutate(x = x + a_1 / 2)) 
    if (nrow(p1) * nrow(p2) == 0)
      c <- 0
    else {
      strip_one <- p1 %>% filter(x > (a_1 / 2) - r)
      strip_two <- p2 %>% filter(x < (a_1  / 2) + r)
      c <- count_r_two(strip_one, strip_two, r)
    }
    if (runif(1) < gamma^c) {
      points <- full_join(p1, p2, by = c("x", "y"))
      ifelse (flip, 
              return(points %>% select(x = y, y = x)),
              return(points))
    } 
  } 
}
\end{verbatim}
\end{programcode}

For PRS, it is necessary to locate the \emph{bad points}, pairs of points that are within $r$ distance of each other and fail a $\gamma$ check.

\begin{programcode}{find\_pairs}
\begin{verbatim}
find_pairs <- function(ppp, r, gamma) {
  n <- nrow(ppp)
  if (n == 0) return(tibble(x = NULL, y = NULL))
  ppp2 <- ppp %>% mutate(k = 1, id = 1:n)  
  pairs <- ppp2 %>% 
    full_join(ppp2, by = "k") %>% 
    filter(id.x < id.y) %>%
    mutate(dist = sqrt((x.x - x.y)^2 + (y.x - y.y)^2))
  pairsu <- pairs %>% mutate(u = runif(nrow(pairs)))
  cpairs <- pairsu %>% filter(dist < r) %>% filter(u > gamma)
  left <- cpairs %>% select(x = x.x, y = y.x)
  right <- cpairs %>% select(x = x.y, y = y.y)
  return(union(left, right))
}
\end{verbatim}
\end{programcode}

This function generates new points around the bad points.
\begin{programcode}{new\_points}
\begin{verbatim}
new_points <- function(bad, s_1, s_2, lambda, r) {
  n <- rpois(1, s_1 * s_2 * lambda)
  if (n == 0) return(tibble(x = NULL, y = NULL))
  poss <- tibble(
    x = runif(n) * s_1,
    y = runif(n) * s_2,
    k = 1,
    id = 1:n
  )
  pairs <- poss %>% 
    full_join(bad %>% mutate(k = 1, id = 1:nrow(bad)), by = "k") %>%
    mutate(dist = sqrt((x.x - x.y)^2 + (y.x - y.y)^2)) %>%
    filter(dist < r) %>%
    select(x = x.x, y = y.x)
  return(distinct(pairs))
}
\end{verbatim}
\end{programcode}

With this, partial rejection sampling for Strauss can be implemented.
\begin{programcode}{strauss\_prs}
\begin{verbatim}
strauss_prs <- function(s_1, s_2, lambda, r, gamma) {
  n <- rpois(1, s_1 * s_2 * lambda)
  ppp <- tibble(
    x = runif(n) * s_1,
    y = runif(n) * s_2
  )
  bad <- find_pairs(ppp, r, gamma)
  while (nrow(bad) > 0) {
    np <- new_points(bad, s_1, s_2, lambda, r) 
    ppp <- union(setdiff(ppp, bad), np)
    bad <- find_pairs(ppp, r, gamma)
  }
  return(ppp)
}
\end{verbatim}
\end{programcode}

\section{Conclusion}

Stitching is a simple to implement algorithm that has an exponential running time with a rate far lower than either acceptance rejection or various local methods.  This enables its use in generating from the Strauss process over parameter values and spaces that were previously not possible in a reasonable amount of time.

%
%
%

\end{document}